\documentclass[12pt]{article}

\usepackage{amsmath, amssymb, amsthm, amsfonts, mathtools, array, xfrac, float, multirow, indentfirst, url, enumerate, comment, fullpage, afterpage, pifont, makecell}

\usepackage[numbers]{natbib}

\providecommand{\noopsort}[1]{}
\usepackage{hyperref}
\usepackage{booktabs, siunitx, xcolor, graphicx, titlesec}

\newtheorem{thm}{Theorem}

\newtheorem{Lem}{Lemma}

\title{Primes between consecutive powers}
\author{Michaela Cully-Hugill\\ School of Science\\ University of New South Wales Canberra \\m.cully-hugill@adfa.edu.au}

\begin{document}
\maketitle
\begin{abstract}
\noindent 
This paper updates the explicit interval estimate for primes between consecutive powers. It is shown that there is least one prime between $n^{155}$ and $(n+1)^{155}$ for all $n\geq 1$. This result is in part obtained with a new explicit version of Goldston's 1983 estimate for the error in the truncated Riemann--von Mangoldt explicit formula. \end{abstract}

\section{Primes in intervals}

There are a variety of results on primes in intervals of the form $(x,x+f(x)]$, for some $f(x) < x$. These results typically hold for sufficiently large $x$, and some explicitly calculate the range of $x$ for which they hold. The latter is particularly useful for bounding gaps between large primes. We have computation on gaps between primes up to $4\cdot 10^{18}$ \cite{O_H_P_14}, so explicit interval estimates are intrumental above this. One of the smallest explicit interval results is from Dudek \cite{Dudek_16p}, with primes between consecutive cubes for all $n\geq \exp(\exp(33.3))$. It appears difficult to substantially extend this range, so we can instead look at results for higher powers. The purpose of this paper is to reduce the $m$ for which we know $(n^m,(n+1)^m)$ contains a prime for all $n\geq 1$. Dudek \cite{Dudek_16p} showed that we can take $m=5\cdot 10^9$, and Mattner \cite{Mattner_17} lowered this to $m=1.5\cdot 10^6$. This is improved to the following.
\begin{thm}\label{powers}
There exists at least one prime in the interval $(n^{155},(n+1)^{155})$ for all $n\geq 1$.
\end{thm} 
\noindent This result can also be used in work on prime-representing functions, e.g. see \cite{Elsholtz_20}. To give more context for Theorem \ref{powers}, the following gives a short summary on the different forms of $f(x)$ which allow $(x,x+f(x)]$ to contain a prime.\footnote{Le projet TME-EMT from Olivier Ramar{\'e} is a very useful resource for these and related results.}

Interval estimates date back to Bertrand's postulate of 1845. He proposed that there should be at least one prime in $(x,2x)$ for all integers $x>1$. This was proved by Chebyshev in 1852. Intervals with $f(x)=Cx$ for $1<C<2$ are the largest in the long-run, but can be the smallest for sufficiently small $x$. These results have been refined in \cite{Schoenfeld_76}, \cite{R_S_2003}, \cite{K_L_14}, and most recently by the author and Lee in \cite{CH_L_21}, with corrections in \cite{CH_L_corr-arX}. From \cite{CH_L_corr-arX}, we know there is at least one prime in $\left( x\left( 1- \Delta^{-1} \right), x \right]$ for all $x\geq x_0$ with $(x_0, \Delta) = (4\cdot 10^{18}, 3.9\cdot 10^{7})$ or $(e^{600}, 2.5\cdot 10^{11})$, among others.

For sufficiently large $x$, the next smallest intervals have $f(x) = C_k x (\log x)^{-k},$ with some integer $k\geq 2$ and constant $C_k$. These intervals can be deduced from a certain type of error estimate for the prime number theorem (PNT), as was done by Trudgian \cite[Cor. 2]{Trudgian_16} and Dusart \cite[Prop. 5.4]{Dusart_18}, among others. For example, Corollary 2 of \cite{Trudgian_16} states that for $k=2$ we can take $C_k = 1/111$ for all $x\geq 2898242$.

The smallest intervals in the long-run have $f(x) = C x^a$, with $C>0$ and $a\in (1/2,1)$. Most of these results use estimates for the Riemann zeta-function $\zeta(s)$ and its zeros. For example, Ingham \cite{Ingham_37} found that we can take $f(x) = x^{\theta+\epsilon}$ with $\theta=(1+4c)/(2+4c)$ if we have $\left| \zeta \left(1/2 +it\right) \right| \leq A t^c$ as $t\rightarrow \infty$ with constants $c$, $A>0$ and sufficiently large $x$. Bourgain \cite{Bourgain_17} showed that we can take $c=13/84 +\epsilon$, the smallest to date, which gives $\theta = 34/55$ in Ingham's method. More recently, the best results have incorporated sieve methods. Iwaniec and Jutila \cite{I_J_79} first used a sieving argument to prove $\theta =5/9$ for sufficiently large $x$. At present, the smallest interval is from Baker, Harman, and Pintz \cite{B_H_P_2001}, of $[x,x+x^{0.525}]$.

The result of \cite{B_H_P_2001} is considered particularly strong because of how close it comes to results which assume the Riemann hypothesis (RH). Cram{\'e}r \cite{Cramer_20} showed that assuming RH gives us primes in $\left( x,x+C \sqrt{x} \log x \right]$ for some $C$ and sufficiently large $x$. Carneiro, Milinovich, and Soundararajan \cite[Thm.~5]{C_M_S_19} give the best explicit version, of $C=22/25$ for $x\geq 4$.


As mentioned, the best explicit result in the long run is Dudek's consecutive cubes \cite{Dudek_16p}. This was proved following Ingham's method in \cite{Ingham_37} with $f(x) = 3x^{\frac{2}{3}}$. It would be similarly possible to take $f(x) = m x^{1-\frac{1}{m}}$, and seek primes between consecutive $m^\text{th}$ powers. This is done in Section \ref{powers_proof}: we refine Dudek's method, and utilise the interval results in \cite{CH_L_corr-arX}, to arrive at Theorem \ref{powers}. Improvements come from using corrected and/or more recent estimates for the zeros of $\zeta(s)$, the PNT, and the Riemann--von Mangoldt explicit formula --- detailed in Section \ref{relevant-results}. We also carry out further optimisation, using additional parameters and numerical optimisation functions in Python. In Section \ref{Goldston} we make explicit an asymptotically better estimate for the error in the truncated Riemann--von Mangoldt explicit formula from Goldston \cite{Goldston_83}. Section \ref{discussion} discusses the relative impact of each of these results on Theorem \ref{powers}, and the room for improvement.

\section{Estimates to be used}\label{relevant-results}

The proof of Theorem \ref{powers} begins with Chebyshev’s functions $$\theta(x)=\sum_{p\leq x} \log p \qquad \text{and} \qquad \psi(x) = \sum_{n\leq x} \Lambda (n),$$ where $\Lambda(n)$ is the von Mangoldt function. Results on the zeros of $\zeta(s)$ are useful for estimating these functions. We will use the most recent estimates for the location and density of the non-trivial zeros (in the ``critical strip" $0<\text{Re}(s)<1$), as well as computational verification of RH. RH states that all non-trivial zeros of $\zeta(\sigma +it)$ have $\sigma = 1/2$, and has most recently been verified over $|t|\leq 3\, 000\, 175\, 332\, 800$ by Platt and Trudgian \cite{P_T-RH_21}. The previous computation from Platt \cite{Platt_17} up to $3.06\cdot 10^{10}$ was used in many of the following results, so will be marked with $H_p$. Otherwise, the largest known ``Riemann height" will be denoted $H_0$.

Above $H_0$, the non-trivial zeros are known to lie outside zero-free regions. The classical region is deduced from Hadamard and de la Vall{\'e}e Poussin's proof of the PNT, and has been made explicit and refined in a number of papers, including \cite{Ford_2002} and \cite{M_T_2015}. Depending on the range of $t$ for which the result is used, some of these estimates are better than others. We will use Ford's, in Theorem 3 of \cite{Ford_2002}. It states that for $|t|\geq 2\cdot 10^{14}$ there are no zeros with $\sigma\geq 1-\nu_1(t)$, where
\begin{align}\label{zerofree_Ford}
\nu_1(t) = \frac{1}{R(|t|) \log |t|}, \qquad R(t) = \frac{J(t)+0.685+0.155\log\log t}{\log t \left(0.04962 - \frac{0.0196}{J(t)+1.15} \right)},
\end{align}
and\footnote{This definition of $J(t)$ has been improved by Hiary \cite{Hiary_16} and corrected as per the comments in Section 2 and footnote 3 of \cite{Patel_21}.} $J(t) = \log(t)/6 + \log\log t + \log(0.77)$. For larger $t$, there is an asymptotically wider region from Korobov \cite{Korobov_1958} and Vinogradov \cite{Vinogradov_1958}. Ford \cite{Ford_2002} also made this explicit, proving that for $c=57.54$ and $|t|\geq 3$ there are no zeros with $\sigma\geq 1-\nu_2(t)$ for
\begin{align}\label{zerofree_F}
\nu_2(t) = \frac{1}{c\log^{2/3}t(\log\log t)^{1/3}}.
\end{align}

Outside the zero-free region, but within the critical strip, there are estimates on the number of zeros up to some $T>0$, denoted $N(T)$. Backlund \cite{Backlund_18} proved that $$N(T) = \frac{T}{2\pi} \log\frac{T}{2\pi e} + \frac{7}{8} + S(T) + O\left( \frac{1}{T} \right),$$ where $S(T)=O(\log T)$. This was made explicit by Rosser \cite{Rosser_1941}, and most recently by Hasanalizade, Shen, and Wong \cite[Cor.~1.2]{H_S_W_22}. It was shown that for $T\geq e$,
\begin{align}\label{R(T)}
\left| N(T) - \frac{T}{2\pi} \log\frac{T}{2\pi e} - \frac{7}{8} \right| \leq a_1\log{T} + a_2\log\log{T} + a_3
\end{align} 
with $a_1= 0.1038$, $a_2=0.2573$, and $a_3 = 9.3675$. Sharper estimates are possible for smaller areas of the critical strip: we can estimate $$N(\sigma,T) = \left| \{\rho=\beta+i\gamma: \zeta(\rho)=0,  0<\gamma<T \text{ and } \sigma < \beta < 1\} \right|.$$  
One of the best explicit estimates is from Kadiri, Lumley, and Ng \cite{K_L_N_2018}. Their result builds on Ramar{\'e}'s \cite{Ramare_2016_L} explicit version of Ingham's zero-density estimate in \cite{Ingham_37}, and is valid for any $\sigma > \frac{1}{2} + \frac{d}{\log H}$, with $d>0$ and $H\in [1002,H_0)$. For any $T\geq H_0$, Kadiri et al. give
\begin{equation} \label{zd_K}
N(\sigma, T) \leq N_1(\sigma, T) = C_1(\sigma) \left( \log(kT) \right)^{2\sigma} (\log T)^{5-4\sigma}T^{\frac{8}{3}(1-\sigma)} + C_2(\sigma) \log^2 T,
\end{equation} for any $k\in \left[ 10^9 H^{-1} ,1\right]$. Using $H_0 = H_p$, Table 1 of \cite{K_L_N_2018} lists values of $C_1$ and $C_2$ for specific $\sigma$, after optimising over several parameters.\footnote{Values for $C_1$ and $C_2$ have been re-calculated in \cite{F_K_S_22} and \cite{J_Y_arxiv}.} Another zero-density estimate was given by Simoni\v{c} \cite{Simonic_20}, of an explicit version of Selberg's zero-density estimate \cite{Selberg_46}. For $\frac{1}{2} \leq \sigma \leq \sigma_0 = \frac{1}{2} + \frac{8}{\log T_0}$, $T\geq 2T_0 \geq 2H_0$, and given constant $C(\sigma_0, T_0)$, we can take
\begin{equation} \label{zd_S}
N(\sigma,T) \leq C(\sigma_0, T_0) T^{1-\frac{1}{4}(\sigma-\frac{1}{2})} \log \frac{T}{2}.
\end{equation}
For $\sigma \in \left[1/2, 37/58 \right]$, (\ref{zd_S}) will be better than (\ref{zd_K}) for sufficiently large $T$.

Zero-density estimates are commonly used in estimates for the PNT. The PNT in terms of $\psi(x)$ can be deduced from the Riemann--von Mangoldt formula: for non-trivial zeros $\rho=\beta+i\gamma$ of $\zeta(s)$, and any $x>1$ not a prime power,
\begin{equation*}
\psi(x) = x - \sum_{\rho} \frac{x^\rho}{\rho} - \log(2\pi) - \frac{1}{2} \log \left(1-\frac{1}{x^2}\right).
\end{equation*}
The sum over $\rho$ is divergent for unordered $\rho$, but can be truncated to write 
\begin{equation}\label{trunc_vonM}
\psi(x) = x - \sum_{|\gamma| < T} \frac{x^\rho}{\rho} + E(x,T),
\end{equation}
with $E(x,T)$ decreasing in $T$. Dudek \cite{Dudek_16p} showed that for half odd integers $x>e^{60}$, we have
\begin{equation}\label{Dudek_psi}
|E(x,T)| \leq \frac{2 x\log^2 x}{T}
\end{equation}
with any $T\in (50, x)$. Goldston \cite{Goldston_83} proved an asymptotically smaller estimate of $$E(x,T) = O\left(\frac{x\log x \log\log x}{T}\right),$$ which can be made explicit. This is done in the following section --- see Theorem \ref{ThmRvM}.

The PNT estimates for $\psi(x)$ can be translated into those for $\theta(x)$ if needed. Costa Pereira \cite[Thm. 5]{Costa_85} (see also Dusart \cite{Dusart_18}) gives the best lower bound on their difference,
\begin{equation}\label{Costa_theta}
\psi(x)-\theta(x) > 0.999x^\frac{1}{2} + x^\frac{1}{3},
\end{equation}
which holds for $x\geq e^{38}$, and Broadbent et al. \cite[Cor. 5.1]{BKLNW_21} give the most recent explicit upper bounds. Of these, we will use the version which holds for all $x\geq e^{1000}$,
\begin{equation}\label{Broadbent_diff}
\psi(x) - \theta(x) < a_1 x^\frac{1}{2} + a_2 x^\frac{1}{3},
\end{equation}
with $a_1= 1+ 1.99986 \cdot 10^{-12}$ and $a_2 = 1 + 1.936 \cdot 10^{-8}$.

\section{An explicit version of Goldston's result}\label{Goldston}

Davenport's exposition in Chapter 17 of \cite{Davenport_2013} (see, in particular, equation (3)) shows that for $T>0$ and $x\geq 2$ which is not a prime power,
\begin{align*}
\psi(x) = \frac{1}{2\pi i} \int_{c-iT}^{c+iT} \left(-\frac{\zeta'(s)}{\zeta(s)} \right) \frac{x^s}{s} ds +  O^{*}\left(\frac{1}{\pi T} \sum_{n=1}^\infty \Lambda(n) \left(\frac{x}{n}\right)^c \left| \log\frac{x}{n} \right|^{-1} \right),
\end{align*}
where, here and hereafter, $O^{*}$ denotes a constant of absolute value not exceeding 1. The integral gives the exact main term and smaller-order terms in the truncated Riemann--von Mangoldt explicit formula (see (\ref{trunc_vonM})). The dominant error terms come from the sum. Goldston \cite{Goldston_83} showed that $E(x,T)$ of (\ref{trunc_vonM}) can be reduced to $$E(x,T)=O\left( \frac{x\log x \log\log x}{T} + \frac{x\log T}{T} + \log x \right),$$ for sufficiently large $x\geq 3$ and $T\geq 3$. This is made explicit in Lemma \ref{loglog}, and will be combined with Dudek's estimate \cite{Dudek_16p} for the integral to prove Theorem \ref{ThmRvM}.

\begin{Lem}\label{loglog}
For half odd integers $x\geq x_K$, and $c=1+1/\log x$, we have $$\sum_{n=1}^\infty \Lambda(n) \left(\frac{x}{n}\right)^c \left| \log\frac{x}{n} \right|^{-1} < M x\log x\log\log x,$$ where pairs of $x_K$ and $M$ are given in Table \ref{table_M}. 
\end{Lem}

\begin{thm}\label{ThmRvM}
For $50<T < x$ and half odd integers $x\geq x_K$ we have $$\psi(x) = x - \sum_{|\gamma| < T} \frac{x^\rho}{\rho} + O^{*}\left( \frac{Kx \log x\log\log x}{T} \right),$$ where pairs of $x_K$ and $K$ are given in Table \ref{table_K}.
\end{thm}

\begin{proof}[Proof of Lemma \ref{loglog}]
Let $x$ be half an odd integer, and $$S(x) = \sum_{n=1}^\infty \Lambda(n) \left(\frac{x}{n}\right)^c \left| \log\frac{x}{n} \right|^{-1},$$ with $\alpha>1$ a parameter. The sum can be split into five parts,
\begin{align*}
\sum_{n=1}^\infty = \sum_{n=1}^{[x/\alpha]} + \sum_{n=[x/\alpha]+1}^{[x]-1} + \sum_{n=[x]}^{[x]+1} + \sum_{n=[x]+2}^{[\alpha x]} + \sum_{n=[\alpha x]+1}^\infty,
\end{align*} denoting each partial sum with $S_i$, consecutively from $i=1$ to 5. The resulting bound on $S(x)$ will be optimised over $\alpha$. Estimates for the first and last sums can be taken straight from \cite{Dudek_16p}, as they are both relatively small compared to the overall bound, with
\begin{align*}
S_1+S_5 &= \frac{e}{\log \alpha} x\log x.
\end{align*}
The estimate for $S_3$ can be just as small: using $\Lambda(n) \leq \log n$ and $[x]=x-\frac{1}{2}$ we have
\begin{align*}
S_3 &\leq \left( \frac{x}{x-\frac{1}{2}} \right)^c \log\left(x-\frac{1}{2}\right) \left| \log \left(\frac{x}{x-\frac{1}{2}}\right) \right|^{-1} + \left( \frac{x}{x+\frac{1}{2}} \right)^c \log \left(x+\frac{1}{2}\right) \left| \log\left(\frac{x}{x+\frac{1}{2}}\right) \right|^{-1}.
\end{align*}
To bound the first term we can use $$\left| \log \left(\frac{x}{x-\frac{1}{2}}\right) \right|^{-1} < \left| \log\left(\frac{x}{x+\frac{1}{2}}\right) \right|^{-1} = \left( \log\left(\frac{x+\frac{1}{2}}{x}\right) \right)^{-1},$$ where $$\log\left(\frac{x+\frac{1}{2}}{x}\right) = \int_{x}^{x+1/2} \frac{1}{t} dt > \frac{1}{2x+1},$$
and for $x\geq e^{100}$ we have $$\left( \frac{x}{x-\frac{1}{2}} \right)^c< 1+10^{-43}.$$ Combining these gives
\begin{align*}
S_3 &< 2(1+10^{-43})(2x+1) \log\left(x+1/2\right) < (4+10^{-20})x \log x.
\end{align*} 

The estimates for $S_2$ and $S_4$ utilise Goldston's method. The improvement largely comes from incorporating the Brun--Titchmarsh theorem for primes in intervals. For $S_2$, the inner sum is over $p^k>x/\alpha$, and as we later need $\alpha<2$ we can write
\begin{align*}
S_2 &= x^c \sum_{1\leq k\leq \frac{\log x}{\log 2}} \sum_{p^k =[x/\alpha]+1}^{[x]-1} \frac{\log p}{p^{kc}} \left| \log\frac{x}{p^k} \right|^{-1} \\
&< \alpha^{c} \log \frac{x}{\alpha} \sum_{1\leq k\leq \frac{\log x}{\log 2}} \sum_{p^k =[x/\alpha]+1}^{[x]-1} \left| \log\frac{x}{p^k} \right|^{-1}.
\end{align*}
Using the Taylor series for $\log(1-u)$ with $|u|<1$,
\begin{align}\label{log-Taylor}
\left| \log \frac{x}{p^k} \right| = - \log\left( 1 - \frac{x-p^k}{x} \right) > \frac{x-p^k}{x}\left(1+\frac{x-p^k}{2x} \right),
\end{align}
which results in
\begin{align*}
S_2 &< \alpha^{c}x \log \frac{x}{\alpha} \sum_{1\leq k\leq \frac{\log x}{\log 2}} \sum_{p^k =[x/\alpha]+1}^{[x]-1} \frac{1}{x-p^k} \left(1+\frac{x-p^k}{2x} \right)^{-1}.
\end{align*}


To estimate the double sum, we will consider the cases $k=1$ and $k\geq 2$ separately. As $x$ is half an odd integer, it can always be placed in $(2m,2m+2)$ for some $m\in \mathbb{Z}^{+}$. When $k=1$, we are summing over primes in $\left[[x/\alpha]+1,[x]-1 \right]$. Since we have $[x]\leq 2m+1$ and $[x/\alpha]+1\geq m+1$ for $1<\alpha<2$, this interval is contained within $[m+1, 2m]$ if we choose $\alpha<2$. In this case,
\begin{align*}
\sum_{p =[x/\alpha]+1}^{[x]-1} \frac{1}{x-p} &\leq \sum_{m < p\leq 2m-1} \frac{1}{2m-p}.
\end{align*}

Let $P(x,y)$ denote the number of primes in $(x-y,x]$, so that 
\begin{align*}
\sum_{m\leq p\leq 2m-1} \frac{1}{2m-p} &= \sum_{n=1}^m \frac{1}{n} \left[ P(2m, n+1) - P(2m, n) \right] \\
&\leq \sum_{2\leq n\leq x/\alpha} \frac{1}{n(n-1)} P(2m, n) + \frac{\alpha}{x} P\left(\frac{2x}{\alpha}, \frac{x}{\alpha}+1\right).
\end{align*}
Montgomery and Vaughan's \cite{M_V_73} version of the Brun--Titchmarsh theorem implies
\begin{equation}\label{BT-MV}
P(x,y)\leq \frac{2y}{\log y}
\end{equation}
for $1<y< x$. With this, and the Euler--Maclaurin formula for the resulting sum, we have
\begin{align*}
\sum_{p =[x/\alpha]+1}^{[x]-1} \frac{1}{x-p} &\leq \sum_{2\leq n\leq \frac{x}{\alpha}} \frac{2}{(n-1)\log n} + 2\left(1+\frac{\alpha}{x}\right) \frac{1}{\log\left(\frac{x}{\alpha}+1\right)}  \\
&\leq \sum_{3\leq n\leq \frac{x}{\alpha}} \frac{2}{n\log n} + \frac{2}{\log 2} + \frac{1}{\log 3} + \frac{2}{\log\left(\frac{x}{\alpha}+1\right) } \\
&\leq 2 \log\log \frac{x}{\alpha} + 3.92 + \frac{1}{\log \frac{x_K}{\alpha}} \left( \frac{\alpha}{x_K} + 2 \right),
\end{align*}
with the last line valid for $x\geq x_K$. Note that it would be possible to reduce the constant term by evaluating more terms in the sum over $n$. For example, if the sum were directly evaluated for $2\leq n\leq 30$, the constant term would drop to 3.69. The constant and lower-order term could also be reduced by using more terms in the Euler--Maclaurin expansion.

For $k\geq 2$, the inner sum can be estimated using an integral, with
\begin{align*}
\sum_{\frac{x}{2} < p^k \leq x-1} \frac{1}{x-p^k} &\leq \int_{\frac{x}{2}}^{x-1} \frac{1}{x-y} d[y^{1/k}] \\
&\leq \frac{1}{k} \int_{\frac{x}{2}}^{x-1} \frac{y^{-\frac{1}{2}}}{x-y} dy = \frac{1}{k\sqrt{x}} \left( \log\left( \frac{1+\sqrt{1-1/x}}{1-\sqrt{1-1/x}}\right) -2\log(1+\sqrt{2}) \right).
\end{align*}
Summing over $k$,
\begin{align*}
\log \frac{x}{\alpha}\sum_{2\leq k\leq \frac{\log x}{\log 2}} \sum_{p^k =[x/\alpha]+1}^{[x]-1} \frac{1}{x-p^k} &< \frac{\log \frac{x}{\alpha}}{\sqrt{x}} \left( \log\left( \frac{1+\sqrt{1-1/x}}{1-\sqrt{1-1/x}}\right) -2\log(1+\sqrt{2}) \right) \sum_{2\leq k\leq \frac{\log x}{\log 2}} \frac{1}{k}  \\
&< \frac{\log^2 x}{\sqrt{x}}\left( \log\left(\frac{\log x}{\log 2}\right) + \gamma + \frac{\log 2}{2\log x} - 1 \right).
\end{align*}
Not only is this bound decreasing with $x$, it is practically negligible for sufficiently large $x$: less than $10^{-200}$ for $x\geq e^{1000}$. Combining the estimates for $k=1$ and $k\geq 2$ gives
\begin{align*}
S_2 &< 2\alpha^c x \log \frac{x}{\alpha} \left( \log\log \frac{x}{\alpha} + 1.96 +  \frac{1}{2\log \frac{x_K}{\alpha}} \left( \frac{\alpha}{x_K} + 2 \right) \right) \\
&< 2M_1 \alpha^c x \log \frac{x}{\alpha} \log\log \frac{x}{\alpha},
\end{align*}
with $$M_1= 1+ \frac{1.96}{\log\log \frac{x_K}{\alpha}}  +  \frac{1}{2\log \frac{x_K}{\alpha}\log\log \frac{x_K}{\alpha}} \left( \frac{\alpha}{x_K} + 2 \right).$$

A similar method can be used for $S_4$, but with $x/p^k <1$ and 
\begin{align*}
\left| \log \frac{x}{p^k} \right| &= \sum_{n=1}^\infty \frac{1}{n} \left( \frac{p^k-x}{p^k} \right)^n > \frac{p^k-x}{p^k}.
\end{align*}
We thus have
\begin{align*}
S_4 &< x^c \sum_{1\leq k\leq \frac{\log \alpha x}{\log 2}} \sum_{p^k =[x]+2}^{[\alpha x]} \frac{\log p}{p^{k(c-1)}(p^k-x)} \\
&< x\log x \sum_{1\leq k\leq \frac{\log \alpha x}{\log 2}} \frac{1}{k} \sum_{p^k =[x]+2}^{[\alpha x]} \frac{1}{p^k-x},
\end{align*}
as the factor $p^{-k(c-1)}\log p$ is decreasing for $p\geq e^{\frac{\log x}{k}}$. For $k=1$, the second sum at most covers primes in $(x+1,\alpha x]$. It can be bounded by a sum over integers $n$, where the $n$ term is included if there is a prime in $(x+n,x+n+1]$. Using (\ref{BT-MV}), we find
\begin{align*}
\sum_{x+1<p \leq \alpha x} \frac{1}{p-x} &\leq \sum_{1\leq n \leq(\alpha-1) x-1} \frac{1}{n} \left[ P(x+n+1,n+1)-P(x+n,n) \right] \\
&\leq \sum_{2\leq n \leq (\alpha-1)x-1} \left(\frac{1}{n-1} - \frac{1}{n}\right) P(x+n,n) + \frac{2(\alpha-1)x}{((\alpha-1)x-1)\log((\alpha-1)x)},
\end{align*}
so that for $x\geq x_K$ we have
\begin{align*}
\sum_{x+1<p \leq \alpha x} \frac{1}{p-x} &\leq \sum_{1\leq n \leq (\alpha-1)x-2} \frac{2}{n\log(n+1)} + \frac{2(\alpha-1)x}{((\alpha-1)x-1)\log((\alpha-1)x)} \\
& \leq 2\log\log((\alpha-1)x) +3.92 + \frac{2(\alpha-1)x_K+3}{(\alpha-1)x_K\log((\alpha-1)x_K)},
\end{align*}
For $k\geq 2$, the inner sum can be estimated with
\begin{align*}
\sum_{x+1 < p^k \leq \alpha x} \frac{1}{p^k-x} &\leq \int_{x+1}^{\alpha x} \frac{1}{y-x} d[y^{1/k}] \\
&\leq \frac{1}{k} \int_{x+1}^{\alpha x} \frac{y^{-\frac{1}{2}}}{y-x} dy = \frac{1}{k\sqrt{x}} \left( 2\log\left( \sqrt{x+1}+\sqrt{x} \right) + \log\left(\frac{\sqrt{\alpha}-1}{\sqrt{\alpha}+1}\right) \right) \\
&\leq \frac{2\log\left( \sqrt{x+1}+\sqrt{x} \right)}{k\sqrt{x}} .
\end{align*}
Thus giving
\begin{align*}
\sum_{2\leq k\leq \frac{\log \alpha x}{\log 2}} \frac{1}{k} \sum_{x+1 < p^k \leq \alpha x} \frac{1}{p^k-x} &\leq \frac{2\log\left( \sqrt{x+1}+\sqrt{x} \right)}{\sqrt{x}} \sum_{2\leq k\leq \frac{\log \alpha x}{\log 2}} \frac{1}{k^2},
\end{align*}
where we can use
\begin{align*}
\sum_{2\leq k\leq \frac{\log \alpha x}{\log 2}} \frac{1}{k^2} \leq \frac{\pi^2}{6} - 1 - \frac{\log 2}{\log 2\alpha x}.
\end{align*}

The overall bound on $S_4$, for $x\geq x_K$, is
\begin{align*}
S_4 &< 2x\log x \left( \log\log((\alpha-1)x) + 1.96 + \frac{(\alpha-1)x_K+3/2}{(\alpha-1)x_K\log((\alpha-1)x_K)} \right) \\
&+ 2\left( \frac{\pi^2}{6} - 1 \right)\sqrt{x}\log x \log(2\sqrt{x+1})  \\
&< 2M_2 x\log x \log\log((\alpha-1)x)
\end{align*}
where $$M_2 = 1 + \frac{1.96\sqrt{x_K} + \left( \frac{\pi^2}{6} - 1 \right)\log(2\sqrt{x_K+1})}{\sqrt{x_K}\log\log((\alpha-1)x_K)} + \frac{2x_K+3(\alpha-1)^{-1}}{2x_K\log((\alpha-1)x_K)\log\log((\alpha-1)x_K)}.$$
The lower bound on $\alpha$ is also needed here, to ensure the last line is true for sufficiently large $x$. It is worth noting that the constant 2 cannot be reduced by increasing the smallest $x$ for which the result holds, as it comes directly from the constant in (\ref{BT-MV}).

Combining the estimates for each $S_i$ gives
\begin{align*}
S(x) &< 2M_1 \alpha^c x \log \frac{x}{\alpha} \log\log \frac{x}{\alpha} +  2M_2 x\log x \log\log((\alpha-1)x) + \left( \frac{e}{\log \alpha} + 4+10^{-20}\right) x \log x  \\
&< M x\log x\log\log x,
\end{align*}
with
\begin{align*}
M= \frac{2M_1 \alpha^c \log \frac{x_K}{\alpha} \log\log \frac{x_K}{\alpha}}{\log x_K\log\log x_K} +  \frac{2M_2\log\log((\alpha-1)x_K)}{\log\log x_K} + \left( \frac{e}{\log \alpha} + 4+10^{-20}\right) \frac{1}{\log\log x_K}
\end{align*}
for all $x\geq x_K$. As $x_K \rightarrow \infty$, $M\rightarrow 2(1+\alpha)$, implying that smaller $\alpha$ is preferable for larger $x$. However, for intermediate values of $x_K$, there will be an optimal value of $\alpha\in (1,2)$. Table \ref{table_M} lists the smallest possible values of $M$ for each $x_K$ after optimising over $\alpha$.
\end{proof}

\begin{table}[H]
\centering
\begin{tabular}{ccccc}
    & \multicolumn{4}{c}{$\log x_K$}  \\ \cline{2-5}
    & $10^3$ & $10^4$ & $10^5$ & $10^6$ \\ \hline 
    $\alpha$   &1.3933 & 1.3501 & 1.3186 & 1.2943 \\ \hline
    $M$    & 7.9074 & 7.1157 & 6.6260 & 6.2904
\end{tabular}
\caption{Admissible values of $M$, with optimised $\alpha$, in Lemma \ref{loglog}.}\label{table_M}
\end{table} 
\nopagebreak[4]

Lemma \ref{loglog} implies that for $x\geq x_K$ we have
\begin{align*}
\psi(x) = \frac{1}{2\pi i} \int_{c-iT}^{c+iT} \left(-\frac{\zeta'(s)}{\zeta(s)} \right) \frac{x^s}{s} ds + O^{*}\left( \frac{M x\log x\log\log x}{\pi T}  \right).
\end{align*}
Cauchy's theorem can be used to evaluate the integral. From \cite{Dudek_16p} we have\footnote{There is a small typo in the bound for $|I_7|$ in \cite{Dudek_16p}, which has been corrected here.} $$|E(x,T)| < \frac{\zeta'(0)}{\zeta(0)} + \frac{1}{2} \log(1-x^{-2}) + 2|I_2| + |I_3| + \frac{M x \log x\log\log x}{\pi T},$$
with, for some positive odd integer $U$, 
\begin{align*}
|I_2| < \ &  \frac{2x\log T}{T-1} + \frac{9+\log \sqrt{U^2+(T+1)^2}}{\pi T x^U} + \frac{9+\log \sqrt{U^2+(T+1)^2}}{2\pi x(T-1)} \\
 &+ \frac{ex\left( \log^2 T + 20\log T \right)}{2\pi (T-1)\log x}  + \frac{e x\log x}{\pi(T-1)} = J_2(x,T,U)
\end{align*}
and $$|I_3| < \frac{9+\log \sqrt{U^2+T^2}}{\pi x^U} = J_3(x,T,U).$$
With these estimates, the error term can be simplified to 
\begin{equation}\label{explicit_Goldston}
|E(x,T)| \leq \frac{K x \log x\log\log x}{T},
\end{equation}
where
\begin{align*}
K= \frac{T}{x_K \log x_K\log\log x_K} \left( \frac{\zeta'(0)}{\zeta(0)} + \frac{\log(1-x_K^{-2})}{2} + 2 J_2(x_K,x_K,U) + J_3(x_K,x_K,U) \right) + \frac{M}{\pi}.
\end{align*}

Although we are free to choose $U$, the size of $x_K$ is such that there is no apparent difference in $K$ as $U$ varies. We expect $J_2$ and $J_3$ to be minimised for small $U$ however, so we chose $U=1$. The following table gives values of $K$ for specific $x_K$.
\begin{table}[H]
\centering
\begin{tabular}{lllll}
    & \multicolumn{4}{c}{$\log x_K$}  \\ \cline{2-5}
    & $10^3$ & $10^4$ & $10^5$ & $10^6$ \\ \hline 
$K$ & 3.4747 & 2.9814 & 2.6821 & 2.4798
\end{tabular}
\caption{Admissible values of $K$ in (\ref{explicit_Goldston}) and Theorem \ref{ThmRvM}.}\label{table_K}
\end{table}

The limiting value of $K$ is $M/\pi$. So, as $x\rightarrow \infty$, $K$ approaches $2(1+\alpha)/\pi$. This could be reduced by refining the estimates for $I_2$ and $I_3$. However, even if $I_2$ and $I_3$ were zero, we find that $K$ only drops to around 2.5 for the estimate over $\log x\geq 1000$.

\section{Primes between consecutive powers}\label{powers_proof}

There will be at least one prime in $(n^m,(n+1)^m)$ for all $n\geq n_0$ if there is a prime in $(x,x+mx^{1-\frac{1}{m}} + \ldots + mx^{\frac{1}{m}} + 1)$ for all $x\geq n_0^m$. Therefore, it will be sufficient to show there is at least one prime in $(x,x+mx^{1-\frac{1}{m}}]$ for sufficiently large $x$. Discarding the smaller-order terms in the upper endpoint does not actually affect the final result, as they become relatively negligible at the values of $x$ we are interested in.

There will be a prime in $(x,x+h]$ if \begin{align} \label{goal}
\theta(x+h)-\theta(x) = \sum_{x< p \leq x+h} \log p
\end{align} 
is positive. This can be translated to $\psi(x)$ with (\ref{Costa_theta}) and (\ref{Broadbent_diff}), in that we have
\begin{align} \label{theta}
\theta(x+h)-\theta(x) > \psi(x+h)-\psi(x) + 0.999\sqrt{x} + x^\frac{1}{3} - a_1\sqrt{x+h} - a_2(x+h)^{1/3}
\end{align}
for $x\geq e^{1000}$, where $a_1$ and $a_2$ are given after (\ref{Broadbent_diff}). Theorem \ref{ThmRvM} gives us the estimate for $\psi(x)$, so for half odd integers $x\geq x_K$ and $50<T<x$ we have
\begin{align} \label{d_psi} 
\psi(x+h) - \psi(x) \geq &\ h - \left| \sum_{|\gamma|<T}  \frac{(x+h)^\rho - x^\rho}{\rho} \right| - K\frac{ G(x,h)}{T}
\end{align}
where $G(x,h) = (x+h)\log(x+h)\log \log(x+h) + x\log x\log\log x$, and values for $K$ and $x_K$ are in Table \ref{table_K}. Between the sum and last term, there should be some optimal value of $T$ which maximises the overall bound. To bound the sum in (\ref{d_psi}), $$\left| \frac{(x+h)^\rho - x^\rho}{\rho} \right| = \left| \int_x^{x+h} u^{\rho -1} du \right| \leq \int_x^{x+h} u^{\beta -1} du\leq h x^{\beta-1}.$$ More terms are possible in this bound, but it would not affect the final result. Hence we have
\begin{equation}\label{D_psi}
\left| \sum_{|\gamma|<T} \frac{(x+h)^\rho - x^\rho}{\rho} \right| \leq h\sum_{|\gamma|< T} x^{\beta-1} .
\end{equation}
The sum can be estimated by writing
\begin{align*}
\sum_{|\gamma|< T} (x^{\beta-1}-x^{-1}) = \sum_{|\gamma|< T} \int_0^\beta x^{\sigma-1} \log x d\sigma = \int_0^1 \sum_{\substack{|\gamma|< T \\ \beta \geq \sigma}} x^{\sigma-1}\log x d\sigma,
\end{align*}
which re-arranges to $$\sum_{|\gamma|< T} x^{\beta-1} = 2\int_0^1 x^{\sigma-1}\log x  \sum_{\substack{0<\gamma< T \\ \beta > \sigma}} 1 \ d\sigma + 2x^{-1} \sum_{0<\gamma < T} 1.$$
Estimates for $N(T)$ and $N(\sigma, T)$ can be used for the two sums. We can also incorporate zero-free regions. For $\nu(T)=\max\{\nu_1(T), \nu_2(T) \}$, defined in (\ref{zerofree_Ford}) and (\ref{zerofree_F}), we have
\begin{align} \label{sum_zeros}
\sum_{|\gamma|< T} x^{\beta-1} = 2x^{-1} N(T) +  \frac{2\log x}{x}\left( \int_0^{1/2} N(T) x^{\sigma} d\sigma  +  \int_{1/2}^{1-\nu(T)} N(\sigma,T) x^{\sigma} d\sigma \right).
\end{align}

The estimate in (\ref{zd_K}) for $N(\sigma, T)$ is asymptotically smaller than (\ref{R(T)}) for $\sigma > 5/8$ and sufficiently large $T$, so it may be useful to use $N(T)$ for some range of $\sigma\in [1/2,1)$. As such, we will re-write the two integrals in (\ref{sum_zeros}) to be split at some $\sigma_1\in [1/2, 1-\nu(T))$, i.e. 
\begin{align} \label{N_int}
&\int_0^{\sigma_1} N(T) x^{\sigma} d\sigma  +  \int_{\sigma_1}^{1-\nu(T)} N_1(\sigma,T) x^{\sigma} d\sigma \nonumber \\
& \quad \leq \frac{T\log T}{2\pi}  \left( \frac{x^{\sigma_1}-1}{\log x} \right) + C_1 T^{8/3}\log^5T  \left( \frac{ W^{1-\nu(T)}-W^{\sigma_1}}{\log W}\right)  + C_2 \log^2T \left( \frac{x^{1-\nu(T)}-x^{\sigma_1}}{\log x} \right),  
\end{align}
where $C_1=C_1(1)$, $C_2=C_2(\sigma_1)$, and $W=x (T^{\frac{4}{3}} \log T)^{-2}.$ Note that the choice of $C_1$ and $C_2$ is because $C_1(\sigma)$ is increasing and $C_2(\sigma)$ is decreasing as per (4.72) and (4.73) in \cite{K_L_N_2018}.

To simplify the bound on (\ref{sum_zeros}), let $T = x^\alpha$ for some $\alpha \in (0,1)$. With (\ref{N_int}) we have
\begin{align*}
\sum_{|\gamma|< T} x^{\beta-1} &< \frac{\alpha\log x}{\pi x^{1-\alpha-\sigma_1}} + 2C_1\alpha^3  \frac{\left(W^{-\nu(x^\alpha)} - W^{\sigma_1-1} \right)\log^4 x }{\log W}  + 2C_2\alpha^2 \log^2 x \left( x^{-\nu(x^\alpha)} - x^{\sigma_1-1} \right).
\end{align*}
The negative term tending to zero as $x\rightarrow \infty$ can be discarded, so that
\begin{align} \label{sum_fns}
\sum_{|\gamma|< T} x^{\beta-1} &< F(x) = \frac{\alpha\log x}{\pi x^{1-\alpha-\sigma_1}} +  2C_1\alpha^3 \frac{\left(W^{-\nu(x^\alpha)} - W^{\sigma_1-1} \right)\log^4 x }{\log W}  +  2C_2\alpha^2 \frac{\log^2 x}{x^{\nu(x^\alpha)} }.
\end{align}
Returning to (\ref{theta}) with the bounds in (\ref{d_psi}) and (\ref{sum_fns}),
\begin{align}\label{theta_2}
\theta(x+h)-\theta(x) > &\ h\left[ 1 - F(x) - K\frac{G(x,h)}{x^\alpha h} + \frac{E(x)}{h}\right],
\end{align}
where $E(x)=0.999x^{1/2} + x^\frac{1}{3} - a_1(x+h)^{1/2} - a_2(x+h)^{1/3}$.

It remains to optimise over $\alpha$ and $\sigma_1$ to find the smallest $m$ satisfying
\begin{align}\label{powers-condition}
1 - F(x) - K \frac{G(x,h)}{x^\alpha h} + \frac{E(x)}{h} >0
\end{align}
for $x\geq x_0$. The LHS of (\ref{powers-condition}) will only be positive and increasing if we take $\alpha> 1/m$ and $\sigma_1 < 1-\alpha$. Of these two parameters, $\alpha$ is far more influential. In fact, the size of $\sigma_1$ has a negligible affect on $F(x)$ for large $x$, as long as the second condition is true. We will use the computation in \cite{O_H_P_14} and intervals in \cite{CH_L_corr-arX} to verify the $m^{\text{th}}$ powers interval for small $x<x_0$, so there is little need to optimise over $\sigma_1$. Hence, we will set $\sigma_1 = 0.6$, to consider any $m\geq 3$.

Values for $C_1$ and $C_2$ can be taken directly from Table 1 of \cite{K_L_N_2018}. However, this table can be updated with the latest Riemann height and a new explicit estimate for the squared divisor function: replacing (3.13) of \cite{K_L_N_2018} with Theorem 2 of \cite{C_T_19}. A correction also needs to be made to Lemma 3.2 of \cite{K_L_N_2018} --- see Remark 1.4 in \cite{F_K_S_22} for more detail. Making these changes, we can take $C_1=17.418$ (computed as in Lemma 2.6 of \cite{J_Y_arxiv}) and $C_2=5.272$. \footnote{$C_2$ was computed using $\{ k, \mu, \alpha, \delta, d \} = \{1, 1.2362, 0.2419, 0.3025, 0.3485 \}$ in \cite[Thm.~1.1]{K_L_N_2018}.}

For interest, the smallest feasible interval result from this method is consecutive cubes. Taking $m=3$ and $\alpha=1/3+10^{-10}$, (\ref{powers-condition}) holds for $x\geq \exp(\exp(33.990))$. This result does hold for all such $x$, despite having only considered half odd integer $x$. The result for those $x$ would imply primes in $(x,X+3X^{2/3}]$ for $X=[x+0.5]+0.5$ and any $x$ in the admissible range. The upper endpoint is always less than $x+3x^\frac{2}{3}+3x^\frac{1}{3}+1$, therefore, we have primes between $n^3$ and $(n+1)^3$ for all $n\geq \exp(\exp(32.892))$. 

An interesting aspect of the result for cubes is that it only uses $\nu_2(T)$ in the zero-free region. We find that $\nu_1(T)$ is a better estimate than $\nu_2(T)$ over $T\leq e^{54594.17..}=\lambda$, and thus becomes useful when we consider larger $m$, as (\ref{powers-condition}) will hold over smaller $x$, $\alpha$, and hence $T$. Because of how $\nu(T)$ is defined, the LHS of (\ref{powers-condition}) decreases to a local minimum at $\lambda^{1/\alpha}$ as $x$ decreases, then increases as the zero-free region switches from using $\nu_2(T)$ to $\nu_1(T)$. If this minimum is still positive, the smallest $x$ for which (\ref{powers-condition}) holds will be a function of $\nu_1(T)$. Therefore, to solve (\ref{powers-condition}) for large $m$, we need an $\alpha>1/m$ for which (\ref{powers-condition}) is positive at $x= \lambda^{1/\alpha}$. Moreover, we want the largest such $\alpha$, as this will maximise the LHS of (\ref{powers-condition}).

We could now find the smallest $m$ for which (\ref{powers-condition}) holds for all $x\geq 1$. However, the intervals in \cite{CH_L_corr-arX} can be smaller than an $m^{\text{th}}$-powers interval for small $x$. This means that (\ref{powers-condition}) need only hold for $x\geq x_0$ if the intervals in \cite{CH_L_corr-arX} verify the $m^{\text{th}}$-powers interval for $x< x_0$. The smallest $m$ for which this was found to work was $m=155$. From Table \ref{table_K}, we can take $K=3.4747$ over $\log x\geq 1000$, and with $\alpha = 0.0080146$ we find that (\ref{powers-condition}) holds for all $\log x\geq e^{4810}$. The intervals in \cite{CH_L_corr-arX} for $x\geq 4\cdot 10^{18}$ and $x\geq e^{1200}$ are smaller than that of consecutive $155^{\text{th}}$ powers for $4\cdot 10^{18} \leq x\leq e^{4850}$, by solving $x\left(1-\Delta^{-1}\right)^{-1} \leq x+mx^{1-1/m}$ for $x$. The computation in \cite{O_H_P_14} verifies the interval for the remaining $x\leq 4\cdot 10^{18}$. Thus, we can say there is at least one prime in $(n^{155}$, $(n+1)^{155})$ for all $n\geq 1$.

\section{Discussion}\label{discussion}

Theorem \ref{powers} is largely determined by the estimates for the zero-free region, zero-counting function, zero-density function, and error term in the truncated Riemann--von Mangoldt explicit formula. Although, some of these estimates are more influential than others. Most notably, a smaller constant in the zero-free region is more likely to affect the results than feasible asymptotic improvements in the zero-density estimate.

Reducing the constants in Kadiri, Lumley, and Ng’s zero-density estimate (\ref{zd_K}) does affect Theorem \ref{powers}, albeit less so for smaller powers. However, to widen the range for which we have primes between consecutive cubes, a smaller power of $T$ in (\ref{zd_K}) would be needed. Simoni\v{c}'s estimate in (\ref{zd_S}) has just this, so it would have been possible to split the second integral in (\ref{sum_zeros}) to incorporate (\ref{zd_S}) over some range of $\sigma\in \left[ 1/2,37/58 \right]$. However, not even $N(\sigma,T)=0$ in this range would have made a difference to Theorem \ref{powers}. This is owing to the choice of $\sigma_1$ in (\ref{powers-condition}): we took $\sigma_1<1-1/m$ to make $F(x)$ of (\ref{sum_fns}) decreasing in the long-run. This makes the terms from the trivial estimate and Simoni\v{c}'s estimate negligible compared to the main term. Therefore, a better zero-density estimate will only be useful for this method if it can be used for $\sigma$ up to the zero-free region.

The zero-free regions are arguably the most influential ingredients in the proof. A smaller constant in either would affect Theorem \ref{powers}. For example, if it were possible to take $c=50$ in (\ref{zerofree_F}), it would give primes between consecutive $150^{\text{th}}$ powers.

Using the explicit version of Goldston's estimate in place of (\ref{Dudek_psi}) affected both Theorem \ref{powers} and the range for consecutive cubes. This suggests it would be worth refining this estimate. The author and Johnston \cite{CH_DJ_arxiv} recently worked on an explicit version of an estimate from Wolke \cite{Wolke_1983} and Ramar{\'e} \cite{Ramare_16_Perron}, which looks likely to improve Theorem \ref{powers}. Theorem 1.2 of \cite{CH_DJ_arxiv} would allow us to take $E(x,T)= O(x/T)$ in (\ref{trunc_vonM}) for $\log x \leq T\leq \sqrt{x}$.

The tactics to cover small $x$ are another important aspect of the result. The intervals in \cite{CH_L_corr-arX} were used to verify the $m^{\text{th}}$ powers interval over small $x$, and allowed a much smaller value for $m$ than if only the condition of (\ref{powers-condition}) were used. If the results of \cite{CH_L_corr-arX} were not as strong, a combination of intervals could have been used, to differentially cover smaller values of $x$. A good option for some mid-range of $x$ would be intervals of the form \begin{equation}\label{log-interval}
\left( x,x+\frac{C_k x}{\log^k x}\right],
\end{equation}
which contain a prime for all $x\geq x_1$, given any positive integer $k$, and $C_k>0$ determined by $k$ and $x_1$. These are implied by PNT estimates of the form
\begin{align} \label{T_P}
\frac{|\psi(x)-x|}{x} \leq A \left( \frac{\log x}{R}\right)^{B} \exp\left(-C \sqrt{\frac{\log x}{R}} \right),
\end{align}
with positive constants $A$, $B$, $C$, and $R$, given explicitly in \cite[Thm.~11]{R_S_62} and \cite[Thm. 1]{P_T_2021}, among others. See, for example, Corollary 5.5 in \cite{Dusart_18}. Say (\ref{powers-condition}) held for all $x\geq x_0$. Then, adjusting $k$ as needed and starting at $x=x_0$, intervals of the form (\ref{log-interval}) could incrementally verify an $m^{\text{th}}$-powers interval for $x_1\leq x\leq x_0$. This method was not needed in the present work, however, because the range covered by the intervals deduced from (\ref{T_P}) did not extend past that of the intervals of \cite{CH_L_corr-arX}.


\subsection*{Acknowledgements}
Many thanks to my supervisor, Tim Trudgian, for his guidance and advice throughout the making of this paper. Thanks also to Thomas Bloom and Aleks Simoni\v{c}, for our discussions on the problem; to DJ, for his insights; to Andrew, for speedy computations; and to the voices of reason and wisdom in 114.

\clearpage

\bibliographystyle{IEEEtranSN}
\bibliography{references_LR}

\begin{thebibliography}{45}
\providecommand{\natexlab}[1]{#1}
\providecommand{\url}[1]{#1}
\csname url@samestyle\endcsname
\providecommand{\newblock}{\relax}
\providecommand{\bibinfo}[2]{#2}
\providecommand{\BIBentrySTDinterwordspacing}{\spaceskip=0pt\relax}
\providecommand{\BIBentryALTinterwordstretchfactor}{4}
\providecommand{\BIBentryALTinterwordspacing}{\spaceskip=\fontdimen2\font plus
\BIBentryALTinterwordstretchfactor\fontdimen3\font minus
  \fontdimen4\font\relax}
\providecommand{\BIBforeignlanguage}[2]{{%
\expandafter\ifx\csname l@#1\endcsname\relax
\typeout{** WARNING: IEEEtranSN.bst: No hyphenation pattern has been}%
\typeout{** loaded for the language `#1'. Using the pattern for}%
\typeout{** the default language instead.}%
\else
\language=\csname l@#1\endcsname
\fi
#2}}
\providecommand{\BIBdecl}{\relax}
\BIBdecl

\bibitem[Backlund(1918)]{Backlund_18}
R.~J. Backlund, ``\"{U}ber die {N}ullstellen der {R}iemannschen
  {Z}etafunktion,'' \emph{Acta Math.}, vol.~41, no.~1, pp. 345--375, 1918.

\bibitem[Baker et~al.(2001)Baker, Harman, and Pintz]{B_H_P_2001}
R.~C. Baker, G.~Harman, and J.~Pintz, ``The difference between consecutive
  primes. {II},'' \emph{Proc. London Math. Soc.}, vol.~83, no.~3, pp. 532--562,
  2001.

\bibitem[Bourgain(2017)]{Bourgain_17}
J.~Bourgain, ``Decoupling, exponential sums and the {R}iemann zeta function,''
  \emph{J. Amer. Math. Soc.}, vol.~30, no.~1, pp. 205--224, 2017.

\bibitem[Broadbent et~al.(2021)Broadbent, Kadiri, Lumley, Ng, and
  Wilk]{BKLNW_21}
S.~Broadbent, H.~Kadiri, A.~Lumley, N.~Ng, and K.~Wilk, ``Sharper bounds for
  the {C}hebyshev function {$\theta(x)$},'' \emph{Math. Comp.}, vol.~90, no.
  331, pp. 2281--2315, 2021.

\bibitem[Carneiro et~al.(2019)Carneiro, Milinovich, and
  Soundararajan]{C_M_S_19}
E.~Carneiro, M.~B. Milinovich, and K.~Soundararajan, ``Fourier optimization and
  prime gaps,'' \emph{Comment. Math. Helv.}, vol.~94, no.~3, pp. 533--568,
  2019.

\bibitem[Costa~Pereira(1985)]{Costa_85}
N.~Costa~Pereira, ``Estimates for the {C}hebyshev function {$\psi(x)-\theta
  (x)$},'' \emph{Math. Comp.}, vol.~44, no. 169, pp. 211--221, 1985.

\bibitem[Cram{\'e}r(1920)]{Cramer_20}
H.~Cram{\'e}r, ``Some theorems concerning prime numbers,'' \emph{Ark. Mat.
  Astr. Fys.}, vol.~5, pp. 1--32, 1920.

\bibitem[Cully-Hugill and Johnston(2021)]{CH_DJ_arxiv}
M.~Cully-Hugill and D.~R. Johnston, ``On the error term in the explicit formula
  of {R}iemann--von {M}angoldt,'' 2021, preprint on arXiv:2111.10001.

\bibitem[Cully-Hugill and Lee(2022{\natexlab{a}})]{CH_L_21}
M.~Cully-Hugill and E.~S. Lee, ``Explicit interval estimates for prime
  numbers,'' \emph{Math. Comp.}, vol.~91, no. 336, pp. 1955--1970, 2022.

\bibitem[Cully-Hugill and Lee(2022{\natexlab{b}})]{CH_L_corr-arX}
------, ``Explicit interval estimates for prime numbers,'' 2022, preprint on
  arXiv:2103.05986.

\bibitem[Cully-Hugill and Trudgian(2021)]{C_T_19}
M.~Cully-Hugill and T.~Trudgian, ``Two explicit divisor sums,'' \emph{Ramanujan
  J.}, vol.~56, pp. 141--149, 2021.

\bibitem[Davenport(1980)]{Davenport_2013}
H.~Davenport, \emph{Multiplicative {N}umber {T}heory}.\hskip 1em plus 0.5em
  minus 0.4em\relax New York: Springer-Verlag, 1980, vol.~74, second edition.

\bibitem[Dudek(2016)]{Dudek_16p}
A.~W. Dudek, ``An explicit result for primes between cubes,'' \emph{Funct.
  Approx. Comment. Math.}, vol.~55, no.~2, pp. 177--197, 2016.

\bibitem[Dusart(2018)]{Dusart_18}
P.~Dusart, ``Explicit estimates of some functions over primes,''
  \emph{Ramanujan J.}, vol.~45, no.~1, pp. 227--251, 2018.

\bibitem[Elsholtz(2020)]{Elsholtz_20}
C.~Elsholtz, ``Unconditional prime-representing functions, following {M}ills,''
  \emph{Amer. Math. Monthly}, vol. 127, no.~7, pp. 639--642, 2020.

\bibitem[Fiori et~al.(2022)Fiori, Kadiri, and Swidinsky]{F_K_S_22}
A.~Fiori, H.~Kadiri, and J.~Swidinsky, ``Density results for the zeros of zeta
  applied to the error term in the prime number theorem,'' 2022, preprint on
  arXiv:2204.02588v1.

\bibitem[Ford(2002)]{Ford_2002}
K.~Ford, ``Zero-free regions for the {R}iemann zeta function,'' in \emph{Number
  theory for the millennium, {II} ({U}rbana, {IL}, 2000)}.\hskip 1em plus 0.5em
  minus 0.4em\relax A K Peters, Natick, MA, 2002, pp. 25--56.

\bibitem[Goldston(1983)]{Goldston_83}
D.~A. Goldston, ``On a result of {L}ittlewood concerning prime numbers. {II},''
  \emph{Acta. Arith.}, vol.~43, no.~1, pp. 49--51, 1983.

\bibitem[Hasanalizade et~al.(2022)Hasanalizade, Shen, and Wong]{H_S_W_22}
\BIBentryALTinterwordspacing
E.~Hasanalizade, Q.~Shen, and P.-J. Wong, ``Counting zeros of the {R}iemann
  zeta function,'' \emph{J. Number Theory}, vol. 235, pp. 219--241, 2022.
  [Online]. Available: \url{https://doi.org/10.1016/j.jnt.2021.06.032}
\BIBentrySTDinterwordspacing

\bibitem[Hiary(2016)]{Hiary_16}
\BIBentryALTinterwordspacing
G.~A. Hiary, ``An explicit van der {C}orput estimate for {$\zeta(1/2+it)$},''
  \emph{Indag. Math. (N.S.)}, vol.~27, no.~2, pp. 524--533, 2016. [Online].
  Available: \url{https://doi.org/10.1016/j.indag.2015.10.011}
\BIBentrySTDinterwordspacing

\bibitem[Ingham(1937)]{Ingham_37}
A.~E. Ingham, ``On the difference between consecutive primes,'' \emph{Q. J.
  Math}, vol.~8, no.~1, pp. 255--266, 1937.

\bibitem[Iwaniec and Jutila(1979)]{I_J_79}
\BIBentryALTinterwordspacing
H.~Iwaniec and M.~Jutila, ``Primes in short intervals,'' \emph{Ark. Mat.},
  vol.~17, no.~1, pp. 167--176, 1979. [Online]. Available:
  \url{https://doi.org/10.1007/BF02385465}
\BIBentrySTDinterwordspacing

\bibitem[Johnston and Yang(2022)]{J_Y_arxiv}
D.~R. Johnston and A.~Yang, ``Some explicit estimates for the error term in the
  prime number theorem,'' 2022, preprint on arXiv:2204.01980v2.

\bibitem[Kadiri and Lumley(2014)]{K_L_14}
H.~Kadiri and A.~Lumley, ``Short effective intervals containing primes,''
  \emph{Integers}, no. A61, pp. 1--18, 2014.

\bibitem[Kadiri et~al.(2018)Kadiri, Lumley, and Ng]{K_L_N_2018}
H.~Kadiri, A.~Lumley, and N.~Ng, ``Explicit zero density for the {R}iemann zeta
  function,'' \emph{J. Math. Anal. Appl.}, vol. 465, no.~1, pp. 22--46, 2018.

\bibitem[Korobov(1958)]{Korobov_1958}
N.~M. Korobov, ``Estimates of trigonometric sums and their applications,''
  \emph{Uspehi Mat. Nauk}, vol.~13, no.~4, pp. 185--192, 1958.

\bibitem[Mattner(2017)]{Mattner_17}
C.~Mattner, ``\textit{Prime numbers in short intervals}. [{H}onours thesis,
  {A}ustralian {N}ational {U}niversity],'' 2017.

\bibitem[Montgomery and Vaughan(1973)]{M_V_73}
H.~L. Montgomery and R.~C. Vaughan, ``The large sieve,'' \emph{Mathematika},
  vol.~20, no.~2, pp. 119--134, 1973.

\bibitem[Mossinghoff and Trudgian(2015)]{M_T_2015}
M.~J. Mossinghoff and T.~S. Trudgian, ``Nonnegative trigonometric polynomials
  and a zero-free region for the {R}iemann zeta-function,'' \emph{J. Number
  Theory}, vol. 157, pp. 329--349, 2015.

\bibitem[Oliveira~e Silva et~al.(2014)Oliveira~e Silva, Herzog, and
  Pardi]{O_H_P_14}
T.~Oliveira~e Silva, S.~Herzog, and S.~Pardi, ``Empirical verification of the
  even {G}oldbach conjecture and computation of prime gaps up to {$4\cdot
  10^{18}$},'' \emph{Math. Comp.}, vol.~83, no. 288, pp. 2033--2060, 2014.

\bibitem[Patel(2020)]{Patel_21}
D.~Patel, ``An {E}xplicit {U}pper {B}ound for $|\zeta(1+it)|$,'' 2020, preprint
  on arXiv:2009.00769v1.

\bibitem[Platt and Trudgian(2021{\natexlab{a}})]{P_T-RH_21}
D.~Platt and T.~Trudgian, ``The {R}iemann hypothesis is true up to {$3\cdot
  10^{12}$},'' \emph{Bull. Lond. Math. Soc.}, vol.~53, no.~3, pp. 792--797,
  2021.

\bibitem[Platt(2017)]{Platt_17}
D.~J. Platt, ``Isolating some non-trivial zeros of zeta,'' \emph{Math. Comp.},
  vol.~86, no. 307, pp. 2449--2467, 2017.

\bibitem[Platt and Trudgian(2021{\natexlab{b}})]{P_T_2021}
D.~J. Platt and T.~S. Trudgian, ``The error term in the prime number theorem,''
  \emph{Math. Comp.}, vol.~90, no. 328, pp. 871--881, 2021.

\bibitem[Ramar\'{e}(2016)]{Ramare_16_Perron}
\BIBentryALTinterwordspacing
O.~Ramar\'{e}, ``Modified truncated {P}erron formulae,'' \emph{Ann. Math.
  Blaise Pascal}, vol.~23, no.~1, pp. 109--128, 2016. [Online]. Available:
  \url{http://ambp.cedram.org/item?id=AMBP_2016__23_1_109_0}
\BIBentrySTDinterwordspacing

\bibitem[Ramar{\'e}(2016)]{Ramare_2016_L}
O.~Ramar{\'e}, ``An explicit density estimate for {D}irichlet ${L}$-series,''
  \emph{Math. Comp.}, vol.~85, no. 297, pp. 325--356, 2016.

\bibitem[Ramar{\'e} and Saouter(2003)]{R_S_2003}
O.~Ramar{\'e} and Y.~Saouter, ``Short effective intervals containing primes,''
  \emph{J. Number Theory}, vol.~98, no.~1, pp. 10--33, 2003.

\bibitem[Rosser(1941)]{Rosser_1941}
B.~Rosser, ``Explicit bounds for some functions of prime numbers,'' \emph{Amer.
  J. Math.}, vol.~63, no.~1, pp. 211--232, 1941.

\bibitem[Rosser and Schoenfeld(1962)]{R_S_62}
\BIBentryALTinterwordspacing
J.~B. Rosser and L.~Schoenfeld, ``Approximate formulas for some functions of
  prime numbers,'' \emph{Illinois J. Math.}, vol.~6, pp. 64--94, 1962.
  [Online]. Available: \url{http://projecteuclid.org/euclid.ijm/1255631807}
\BIBentrySTDinterwordspacing

\bibitem[Schoenfeld(1976)]{Schoenfeld_76}
L.~Schoenfeld, ``Sharper bounds for the {C}hebyshev functions $\theta(x)$ and
  $\psi(x)$. {II},'' \emph{Math. Comp.}, vol.~30, no. 134, pp. 337--360, 1976.

\bibitem[Selberg(1946)]{Selberg_46}
A.~Selberg, ``Contributions to the theory of the {R}iemann zeta-function,''
  \emph{Arch. Math. Naturvid.}, vol.~48, no.~5, pp. 89--155, 1946.

\bibitem[Simoni\v{c}(2020)]{Simonic_20}
A.~Simoni\v{c}, ``Explicit zero density estimate for the {R}iemann
  zeta-function near the critical line,'' \emph{J. Math. Anal. Appl.}, vol.
  491, no.~1, pp. 124\,303, 41, 2020.

\bibitem[Trudgian(2016)]{Trudgian_16}
T.~Trudgian, ``Updating the error term in the prime number theorem,''
  \emph{Ramanujan J.}, vol.~39, no.~2, pp. 225--234, 2016.

\bibitem[Vinogradov(1958)]{Vinogradov_1958}
I.~M. Vinogradov, ``A new estimate of the function $\zeta (1+it)$,'' \emph{Izv.
  Akad. Nauk SSSR. Ser. Mat.}, vol.~22, no.~2, pp. 161--164, 1958.

\bibitem[Wolke(1983)]{Wolke_1983}
D.~Wolke, ``On the explicit formula of {R}iemann-von {M}angoldt, {II},''
  \emph{J. London Math. Soc.}, vol.~2, no.~3, pp. 406--416, 1983.

\end{thebibliography}

\end{document}